\documentclass[10pt]{amsart}
\usepackage{enumerate,amsmath,amssymb,latexsym,
amsfonts, amsthm, amscd, MnSymbol}

\theoremstyle{plain}

\newtheorem{theorem}{Theorem}

\newtheorem*{theorem*}{Theorem}

\numberwithin{equation}{section}

\newcommand{\R}{\mathbb{R}}

\newcommand{\sn}{{\rm sn}}

\newcommand{\ii}{{\rm i}}

\newcommand{\ddd}{{\rm dn}_3}

\begin{document}

\title {The Berndt-Bhargava-Garvan Transfer Principle for signature three}

\date{}

\author[P.L. Robinson]{P.L. Robinson}

\address{Department of Mathematics \\ University of Florida \\ Gainesville FL 32611  USA }

\email[]{paulr@ufl.edu}

 \keywords{}

\begin{abstract}
We offer a new proof for the Berndt-Bhargava-Garvan Transfer Principle that connects the signature-three elliptic theory of Ramanujan to the classical elliptic theory. 
\end{abstract}

\maketitle

\section*{Introduction}

\medbreak 

In [1995] Berndt, Bhargava and Garvan introduce a Transfer Principle by means of which to pass between `classical' elliptic function theory and elliptic function theory in signature three. Their transfer principle is a direct consequence of then-new transformation formulae that relate the hypergeometric functions $F_2 = F(\tfrac{1}{2}, \tfrac{1}{2}; 1; \bullet)$ and $F_3 = F(\tfrac{1}{3}, \tfrac{2}{3}; 1; \bullet)$: explicitly, if with $0 < p < 1$ we define 
$$\alpha = \frac{p^3 (2 + p)}{1 + 2p} \; \; \; {\rm and} \; \; \; \beta = \frac{27}{4} \frac{p^2 (1+p)^2}{(1 + p + p^2)^3}$$
then 
$$(1 + p + p^2) \, F_2 (\alpha) = \sqrt{1 + 2p} \, F_3 (\beta)$$
and 
$$(1 + p + p^2) \, F_2 (1 - \alpha) = \sqrt{3 + 6p} \, F_3 (1 - \beta)$$
whence 
$$\frac{ F_2 (1 - \alpha)}{F_2 (\alpha)} = \sqrt3 \, \frac{F_3 (1 - \beta)}{F_3 (\beta)}.$$ In [1995] these hypergeometric identities appear as Theorem 5.6, Corollary 5.7 and Corollary 5.8 respectively; the transfer principle itself appears as Theorem 5.9. The proofs in [1995] rest on calculations involving $q$-series, especially the cubic theta-functions $a(q), b(q)$ and $c(q)$ of the brothers Borwein. 

\medbreak 

Here, we present completely different proofs of these hypergeometric identities. Our proofs derive from the elliptic function $\ddd$ of Shen [2004]; they therefore serve to reinforce the place of $\ddd$ within the theory of elliptic functions in signature three. In Section 1 we review the construction of the elliptic function $\ddd$ and its coperiodic Weierstrass function $p$; further, we express the fundamental periods of the elliptic functions $\ddd$ and $p$ explicitly in terms of the `signature-three' hypergeometric function $F(\tfrac{1}{3}, \tfrac{2}{3}; 1; \bullet)$. In Section 2 we review the classical Jacobian perspective on the Weierstrass function $p$; this facilitates explicit expressions for the fundamental periods of $\ddd$ in terms of the `classical' hypergeometric function $F(\tfrac{1}{2}, \tfrac{1}{2}; 1; \bullet)$. In Section 3 we assemble the pieces to deduce the aforementioned hypergeometric identities.

\medbreak 

\section{The elliptic function $\ddd$}

\medbreak 

The function $\ddd$ is initially defined as a strictly-increasing function from $\R$ onto $\R$; this function is then seen to satisfy a differential equation whose solutions are known to be elliptic. Here, we outline the construction; for details, see [2004]. 

\medbreak 

We begin by fixing $0 < \kappa < 1$ as modulus and $\lambda = \sqrt{1 - \kappa^2}$ as complementary modulus. We define $\delta_{\kappa} : \R \to \R$ to be the derivative of the function that is inverse to 
$$T \mapsto \int_0^T F(\tfrac{1}{3}, \tfrac{2}{3} ; \tfrac{1}{2} ; \kappa^2 \sin^2 t) \, {\rm d}t.$$ 
The function $\delta_{\kappa}$ (easily) satisfies the initial condition $\delta_{\kappa} (0) = 1$ and (less easily) satisfies the differential equation 
$$9 (\delta_{\kappa}')^2 = 4 (1 - \delta_{\kappa}) (\delta_{\kappa}^3 + 3 \delta_{\kappa}^2 - 4 \lambda^2).$$ 
The solution of this initial value problem is readily identified, as follows. 

\medbreak 

\begin{theorem} \label{dn3}
The function $\delta_{\kappa}$ satisfies  
$$(1 - \delta_{\kappa}) (\tfrac{1}{3} + p_{\kappa}) = \tfrac{4}{9} \kappa^2$$ 
where $p_{\kappa} = \wp(\bullet; g_2, g_3)$ is the Weierstrass function with invariants 
$$g_2 = \tfrac{4}{27} (9 - 8 \kappa^2) = \tfrac{4}{27} (8 \lambda^2 + 1)$$ 
and 
$$g_3 = \tfrac{8}{729} (27 - 36 \kappa^2 + 8 \kappa^4) = \tfrac{8}{729} (8 \lambda^4 + 20 \lambda^2 - 1).$$ 
\end{theorem} 

\begin{proof} 
See [2004]. 
\end{proof} 

\medbreak 

The elliptic function $\ddd$ of Shen is the elliptic extension of $\delta_{\kappa}$ whose existence is guaranteed by this Theorem: thus, 
$$\ddd = 1 - \frac{\tfrac{4}{9} \kappa^2}{\tfrac{1}{3} + p_{\kappa}}\, .$$

\medbreak 

The elliptic function $\ddd$ and the Weierstrass function $p_{\kappa}$ are plainly coperiodic. We write $(2 \omega_{\kappa}, 2 \omega_{\kappa} ')$ for their fundamental pair of periods with $\omega_{\kappa}$ and $- \ii \omega_{\kappa} '$ strictly positive. 

\medbreak 

A virtue of this construction is that it provides immediate access to the real half-period $\omega_{\kappa}$ in explicit hypergeometric terms. 

\medbreak 

\begin{theorem} \label{omega}
The real half-period $\omega_{\kappa}$ of $\ddd$ and $p_{\kappa}$ is given by 
$$\omega_{\kappa} = \tfrac{1}{2} \pi  F(\tfrac{1}{3}, \tfrac{2}{3} ; 1 ; \kappa^2).$$
\end{theorem} 

\begin{proof} 
From the definition of $\delta_{\kappa}$, its least positive period is equal to twice the integral 
$$\int_0^{\frac{1}{2} \pi} F(\tfrac{1}{3}, \tfrac{2}{3} ; \tfrac{1}{2} ; \kappa^2 \sin^2 t) \, {\rm d}t;$$ 
this may be calculated by expanding the hypergeometric series and integrating termwise, with the result announced in the Theorem. 
\end{proof} 

\medbreak 

An explicit hypergeometric expression for the imaginary half-period $\omega_{\kappa}'$ lies a little deeper. We bring it to light by applying to the Weierstrass function $p_{\kappa} = \wp(\bullet; g_2, g_3) = \wp (\bullet; \omega_{\kappa}, \omega_{\kappa} ')$ the modular transformation according to which its imaginary period is divided by three: thus, we introduce the Weierstrass function 
$$q_{\kappa} = \wp (\bullet; \omega_{\kappa}, \tfrac{1}{3} \omega_{\kappa} ').$$ 

\medbreak 

\begin{theorem} \label{q}
The Weierstrass functions $q_{\kappa}$ and $p_{\lambda}$ are related by  
$$q_{\kappa} (z) = - 3 \, p_{\lambda} (\sqrt3 \, \ii \, z).$$ 
\end{theorem} 

\begin{proof} 
Let the Weierstrass function $q_{\kappa}$ have quadrinvariant $h_2$ and cubinvariant $h_3$: then 
$$h_2 = 120 b^2 - 9 g_2$$
and 
$$h_3 = 280 b^3 - 42 b g_2 - 27 g_3$$
where $b$ is the value of $p_{\kappa}$ at $\tfrac{2}{3} \omega_{\kappa} '$; for this general consequence of Weierstrassian trimidiation, see Section 68 of [1973]. It is proved in [2004] Section 5 that $p_{\kappa} (\tfrac{2}{3} \omega_{\kappa} ')$   is precisely $- \tfrac{1}{3}$; using also $g_2$ and $g_3$ from Theorem \ref{dn3} it follows that 
$$h_2 = \tfrac{4}{3} (1 + 8 \kappa^2)$$ 
and 
$$h_3 = \tfrac{8}{27} (1 - 20 \kappa^2 - 8 \kappa^4).$$ 
A glance at Theorem \ref{dn3} reveals the relationship between these invariants and those of $p_{\lambda}$: namely, $h_2$ is $9 = (\sqrt3 \ii)^4$ times the quadrinvariant of $p_{\lambda}$ and $h_3$ is $-27 = (\sqrt3 \ii)^6$ times the cubinvariant of $p_{\lambda}$. Finally, the homogeneity relation for $\wp$-functions completes the proof. 

\end{proof} 

\medbreak 

Here, notice both the switch to the complementary modulus and the quarter-rotation of the period lattice. 

\medbreak 

We are now prepared to identify the imaginary half-period $\omega_{\kappa}'$ in explicit hypergeometric terms. 

\medbreak 

\begin{theorem} \label{omega'}
$$\omega_{\kappa} ' = \ii \, \tfrac{\sqrt3}{2} \pi  F(\tfrac{1}{3}, \tfrac{2}{3} ; 1 ; 1 - \kappa^2).$$
\end{theorem} 

\begin{proof} 
Upon comparing the fundamental half-periods of the Weierstrass functions involved, we see at once that Theorem \ref{q} implies the relation 
$$\omega_{\kappa}' = \ii \, \sqrt3 \, \omega_{\lambda};$$
the present Theorem therefore follows from Theorem \ref{omega} in view of the fact that $\lambda^2 = 1 - \kappa^2.$ 
\end{proof} 

\medbreak 

\section{The Jacobian reformulation}

\medbreak 

Here, we derive alternative explicit expressions for the fundamental periods of $\ddd$ in terms of the hypergeometric function $F(\tfrac{1}{2}, \tfrac{1}{2}; 1; \bullet)$ that is associated to the classical theory of Jacobian elliptic functions. 

\medbreak 

Recall (from Chapter XXII of [1927] for instance) the fact that if $p$ is a Weierstrass function with real midpoint values $e_1 > e_2 > e_3$ then 
$$p(z) = e_3 + \frac{e_1 - e_3}{\sn^2 ( z \, (e_1 - e_3)^{1/2})}$$ 
where $\sn = \sn (\bullet, k)$ is the Jacobian elliptic function with modulus $k$ given by 
$$k^2 = \frac{e_2 - e_3}{e_1 - e_3}.$$ 
Recall also that $\sn$ has fundamental pair of periods $(4 K, 2 \ii K')$ where 
$$K = \tfrac{1}{2} \pi F(\tfrac{1}{2}, \tfrac{1}{2}; 1; k^2)$$
and 
$$K' = \tfrac{1}{2} \pi F(\tfrac{1}{2}, \tfrac{1}{2}; 1; 1 - k^2);$$ 
recall further that addition of $2 K$ to the argument of $\sn$ effects merely a reversal of sign, so that $\sn^2$ has $(2 K, 2 \ii K')$ as a fundamental pair of periods. It follows that the Weierstrass function $p$ has as fundamental pair of half-periods  
$$\omega = \frac{K}{(e_1 - e_3)^{1/2}} \; \; \; {\rm and} \; \; \; \omega ' = \ii  \frac{K'}{(e_1 - e_3)^{1/2}}.$$ 

\medbreak 

We now elaborate upon these facts as they apply to the Weierstrass function $p_{\kappa}$ that is coperiodic with $\ddd$. Let the (acute) modular angle $\theta$ be defined by 
$$\kappa = \sin \theta$$
and introduce the abbreviations 
$$s = \sin \tfrac{1}{3} \theta \; \; \; {\rm and} \; \; c = \cos \tfrac{1}{3} \theta$$ 
so that trigonometric trimidiation yields 
$$\kappa = s (3 - 4 s^2).$$ 
Factorizing the right-hand side  of the differential equation 
$$(p_{\kappa} ')^2 = 4 (p_{\kappa})^4 - g_2 p_{\kappa} - g_3$$ 
with $g_2$ and $g_3$ as in Theorem \ref{dn3} reveals that the midpoint values of $p_{\kappa}$ are given by 
$$e_1 = \tfrac{2}{9} (8 s^4 - 12 s^2 + 3)$$
$$e_2 = \tfrac{1}{9} (- 8 s^4 + 12 s^2 - 3 + 8 \sqrt3 \, s^3 c)$$ 
$$e_3 = \tfrac{1}{9} (- 8 s^4 + 12 s^2 - 3 - 8 \sqrt3 \, s^3 c).$$ 

\medbreak 

In terms of the foregoing choices of notation, these deliberations have the following outcome. 

\medbreak 

\begin{theorem} \label{classical}
The fundamental half-periods of $p_{\kappa}$ and $\ddd$ are given by 
$$r \, \omega_{\kappa} = \tfrac{1}{2} \pi \,  F(\tfrac{1}{2}, \tfrac{1}{2}; 1; k^2)$$
and 
$$r\, \omega_{\kappa} ' = \ii \, \tfrac{1}{2} \pi \, F(\tfrac{1}{2}, \tfrac{1}{2}; 1; 1 - k^2)$$
where $r > 0$ is given by 
$$r^2 = \tfrac{1}{3 \sqrt3} (8 s^3 c + \sqrt3 (8 s^4 - 12 s^2 + 3))$$ 
and $k > 0$ is given by 
$$k^2 = \frac{16 s^3 c}{8 s^3 c + \sqrt3 (8 s^4 - 12 s^2 + 3)}.$$ 

\end{theorem} 

\begin{proof} 
First substitute the midpoint values into $r = (e_1 - e_3)^{1/2}$ and $k^2 = (e_2 - e_3)/(e_1 - e_3)$; then invoke the relationship between Weierstrassian periods and Jacobian periods that was recalled above.  
\end{proof} 

\medbreak 

\section{The hypergeometric identities}

\medbreak 

We are now in a position to address the hypergeometric identities to which we alluded in our Introduction. In fact, these hypergeometric identities will follow at once from from a direct comparison of the formulae in Theorem \ref{omega} and Theorem \ref{omega'} with the formulae in Theorem \ref{classical} once we introduce the parameter $p$ from [1995]. 

\medbreak 

To introduce this parameter, we begin by noting that the rule 
$$(0, 1) \to (0, 1) : p \mapsto \sqrt{\frac{3 p^2}{1 + p + p^2}}$$ 
defines a bijection; we may coordinate this with the bijection 
$$(0, \tfrac{1}{2} \pi) \to (0, 1): \theta \mapsto 2 s = 2 \sin \tfrac{1}{3} \theta$$
to obtain a parametrization according to which $s = \sin \tfrac{1}{3} \theta$ is given by 
$$s = \frac{\sqrt3}{2} \frac{p}{(1 + p + p^2)^{1/2}}$$
and the complementary $c = \cos \tfrac{1}{3} \theta$ is given by 
$$c = \frac{1}{2} \frac{2 + p}{(1 + p + p^2)^{1/2}}$$ 
while inversely  
$$p = 2 \, \frac{s^2 + \sqrt3 s c}{3 - 4s^2}.$$ 
In terms of this parametrization, the original modulus $\kappa = \sin \theta$ of $\ddd$ is given by
$$\kappa^2 = s^2 (3 - 4 s^2)^2 = \frac{27}{4} \frac{p^2 (1 + p)^2}{(1 + p + p^2)^3} .$$  

\medbreak 

Now return to the context of Theorem \ref{classical}. A moderate amount of calculation confirms that 
$$8 s^4 - 12 s^2 + 3 = \frac{3}{2 (1 + p + p^2)^2}\big(2 + 4 p - 2 p^3 - p^4 \big)$$ 
and 
$$s^3 c = \frac{3 \sqrt3}{16} \frac{p^3 (2 + p)}{(1 + p + p^2)^2}$$ 
whence (after further calculation) the spread of the midpoint values is given by 
$$r^2 = e_1 - e_3 = \frac{1 + 2 p}{(1 + p + p^2)^2}$$ 
and the Jacobian modulus is given by 
$$k^2 = \frac{e_2 - e_3}{e_1 - e_3} = \frac{p^3 (2 + p)}{1 + 2 p}.$$

\medbreak 

As in the Introduction, for typographical convenience we shall adopt the abbreviations 
$$F_2 = F(\tfrac{1}{2}, \tfrac{1}{2}; 1; \bullet)$$ 
and 
$$F_3 = F(\tfrac{1}{3}, \tfrac{2}{3}; 1; \bullet)$$
in each of the following three Theorems. We shall also adopt from [1995] the abbreviations 
$$\alpha = \frac{p^3 (2 + p)}{1 + 2p}$$
and 
$$\beta = \frac{27}{4} \frac{p^2 (1+p)^2}{(1 + p + p^2)^3}.$$
Notice that $\alpha$ is precisely $k^2$ and $\beta$ is none other than $\kappa^2$.

\medbreak 

The following result is [1995] Theorem 5.6. 

\medbreak 

\begin{theorem} \label{5.6} 
If $0 < p < 1$ then 
$$(1 + p + p^2) \, F_2 (\alpha) = \sqrt{1 + 2p} \, F_3 (\beta)$$
\end{theorem} 

\begin{proof} 
Simply compare the formula for $\omega_{\kappa}$ in Theorem \ref{omega} with the formula for $\omega_{\kappa}$ in Theorem \ref{classical}, taking note of how the original modulus $\kappa$, the Jacobian modulus $k$ and the spread $r^2 = e_1 - e_3$ depend on the parameter $p$ as displayed leading up to the present Theorem. 
\end{proof} 

\medbreak 

The following result is [1995] Corollary 5.7. 

\medbreak 

\begin{theorem} \label{5.7}
If $0 < p < 1$ then 
$$(1 + p + p^2) \, F_2 (1 - \alpha) = \sqrt{3 + 6p} \, F_3 (1 - \beta)$$
\end{theorem} 

\begin{proof} 
Simply compare the formula for $\omega_{\kappa} '$ in Theorem \ref{omega'} with the formula for $\omega_{\kappa} '$ in Theorem \ref{classical}, again taking note of how $\kappa$, $k$ and $r^2$ depend on the parameter $p$. 
\end{proof} 

\medbreak 

The following result is [1995] Corollary 5.8 (mildly rewritten).

\medbreak 

\begin{theorem} \label{5.8}
If $0 < p < 1$ then
$$\frac{F_2 (1 - \alpha)}{F_2 (\alpha)} = \sqrt3 \, \frac{F_3 (1 - \beta)}{F_3 (\beta)}.$$
\end{theorem} 

\begin{proof} 
An immediate consequence of Theorem \ref{5.6} and Theorem \ref{5.7}. 
\end{proof} 

\medbreak

\bigbreak

\begin{center} 
{\small R}{\footnotesize EFERENCES}
\end{center} 
\medbreak

[1927] E.T. Whittaker and G.N. Watson, {\it A Course of Modern Analysis}, Fourth Edition, Cambridge University Press. 

\medbreak 

[1973] P. Du Val, {\it Elliptic Functions and Elliptic Curves}, L.M.S. Lecture Note Series {\bf 9}, Cambridge University Press. 

\medbreak 

[1995] B.C. Berndt, S. Bhargava, and F.G. Garvan, {\it Ramanujan's theories of elliptic functions to alternative bases}, Transactions of the American Mathematical Society {\bf 347} 4163-4244. 

\medbreak 

[2004] Li-Chien Shen, {\it On the theory of elliptic functions based on $_2F_1(\frac{1}{3}, \frac{2}{3} ; \frac{1}{2} ; z)$}, Transactions of the American Mathematical Society {\bf 357} 2043-2058.

\medbreak

\end{document}